\documentclass{amsart}
\usepackage{amsfonts}
\usepackage{color}

\usepackage{graphicx}
\usepackage{enumerate}

\usepackage{color}

\usepackage[bookmarksnumbered,colorlinks, linkcolor=blue, citecolor=red, pagebackref, bookmarks, breaklinks]{hyperref}

\def\div{\operatorname {\text{div}}}

\newcommand{\R}{\mathbb{R}}

\newcommand{\N}{\mathbb{N}}

\newtheorem{theorem}{Theorem}[section]
\newtheorem{lemma}[theorem]{Lemma}
\newtheorem{prop}[theorem]{Proposition}

\theoremstyle{definition}

\theoremstyle{remark}

\numberwithin{equation}{section}

\newcommand{\lam}{\lambda}
\begin{document}

\title[Optimization problems with volume constraint]{
Nonstandard growth optimization problems with volume constraint}

\author[ A.  Salort, B. Schvager, A. Silva, ]{Ariel Salort, Belem Schvager and Anal\'ia Silva}

\address[A. Salort]{Instituto de C\'alculo (UBA - CONICET) and Departamento de Matematica, FCEyN, Universidad de Buenos Aires, Pabellon I, Ciudad Universitaria (1428), Buenos Aires, Argentina.}

\email{asalort@dm.uba.ar} \urladdr{http://mate.dm.uba.ar/~asalort/}
\address[A. Silva]{Instituto de Matem\'atica Aplicada San luis (IMASL),
Universidad Nacional de San Luis, CONICET. Ejercito de los Andes
950, D5700HHW , San Luis,
Argentina.}

\email{acsilva@unsl.edu.ar}
\urladdr{https://analiasilva.weebly.com}

\address[B.B. Schvager ]{Instituto de Matem\'atica Aplicada San luis (IMASL),
Universidad Nacional de San Luis, CONICET. Ejercito de los Andes
950, D5700HHW , San Luis,
Argentina.}

\email{bbschvager@unsl.edu.ar}



\begin{abstract}

In this article we study some optimal design problems related to nonstandard growth eigenvalues ruled by the $g-$Laplacian operator. More precisely, given $\Omega\subset \R^n$ and  $\alpha,c>0$ we consider the optimization problem $\inf \{ \lambda_\Omega(\alpha,E)\colon E\subset \Omega, |E|=c \}$, where $\lambda_\Omega(\alpha,E)$ is related to the first eigenvalue to
$$
-\div(g( |\nabla u |)\tfrac{\nabla u}{|\nabla u|}) + g(u)\tfrac{u}{|u|}+ \alpha \chi_E g(u)\tfrac{u}{|u|} \quad \text{ in }\Omega
$$
subject to Dirichlet, Neumann or Steklov boundary conditions.
\\
We analyze existence of an optimal configuration, symmetry properties of them, and the asymptotic behavior as $\alpha$ approaches $+\infty$.
\normalcolor
\end{abstract}

\maketitle

\section{Introduction}
The literature on optimization problems is very wide, from the classical cases of isoperimetrical problems to the most recent applications including elasticity and spectral optimization. Only to mention some references and motivations, we refer the reader to the books of Allaire \cite{All}, Bucur and Buttazzo \cite{BB}, Henrot \cite{He}, Pironneau \cite{Piro} and Sokolowski and Zol\'esio \cite{SZ}, where a huge amount of shape optimization problems is introduced. Optimization problems in the more general form can be stated as follows: given a cost functional $\mathcal{F}$ and a class of admissible domains $\mathcal{E}$, solve the minimization problem
\begin{equation} \label{opti}
\min \left\{ \mathcal{F}(E) \colon E \in \mathcal{E} \right\}.
\end{equation}
In recent years there has been an increasing amount of interest in optimization problems for power-like functionals, see \cite{All, BB,CC,CEP, DPFB, DPFB2,DPFBR,EF, EL,He, M,Piro, SZ} for instance. Moreover, optimization problems describing non-local phenomena have been approached recently \cite{DG,DPFBR, FBRS, PS}. However, optimization problems of the form  \eqref{opti}, where the \emph{state equation} to be solved on $E$ involves  behaviors more general than powers are less common in the literature. Additional drawbacks can arise in these class of problems due to the possible lack of homogeneity of the functional. We cite for instance the articles \cite{DSSSS, SM, MW, ZZZ}.

In this manuscript we  study the existence of an optimal configuration for a minimization problem ruled by the nonlinear degenerate and possibly not homogeneous operator \emph{$g$-Laplacian} defined as $\Delta_g u :=\div( g(|\nabla u|) \frac{\nabla u}{|\nabla u|})$, where $G(t)=\int_0^tg(s)\,ds$ is a Young function fulfilling the following growth condition
\begin{equation} \label{cond} \tag{L}
1<p^- G(t) \leq tg(t) \leq p^+ G(t) <\infty \qquad \text{for all } t\in \R^+,
\end{equation}
for fixed constants $p^+$ and $p^-$ and
\begin{equation} \label{cond2} \tag{C}
 g(t) \text{ is convex   for all } t\geq 0.
\end{equation}

These kind of problems appears naturally when studying general optimal design problems, and they are usually formulated as problems of minimization of the energy, stored in the design under a prescribed loading. Solutions of these problems are unstable to perturbations of the loading, and the optimal design problem is formulated as minimization of the stored energy under the most unfavorable loading. This most dangerous loading is one that maximizes the stored energy over the class of admissible functions. Finally, the problem is reduced to minimization involving the behavior on the boundary. See \cite{CC,SLO} for more  details.


We describe now our problem.  We consider a bounded connected Lipschitz domain $\Omega\subset \R^n$, $\alpha>0$ and  numbers $c\in [0,|\Omega|]$. For any measurable set $E\subset \Omega$,    we consider the quantity  $\lam_\Omega (\alpha,E)$  
\begin{equation} \label{lambda.d} \tag{$\mathcal{D}$}
\lambda_\Omega (\alpha,E)=\inf\left\{ \int_\Omega G(|\nabla u|)\,dx +\alpha\int_E G(|u|)\,dx \colon u\in \mathcal{W}(\Omega)\right\},
\end{equation}
where $\mathcal{W}(\Omega)$ is the set of functions $u\in W^{1,G}_0(\Omega)$ such that $\int_\Omega G(|u|)\,dx =1$. We also consider the quantity
\begin{equation} \label{lambda.n} \tag{$\mathcal{N}$}
\lambda_\Omega (\alpha,E)=\inf\left\{ \int_\Omega G(|\nabla u|) +G(|u|)\,dx+\alpha\int_E G(|u|)\,dx \colon u\in \mathcal{W}(\Omega)\right\},
\end{equation}
where in this case $\mathcal{W}(\Omega)$ is the set of functions $u\in W^{1,G}(\Omega)$ such that $\int_\Omega G(|u|)\,dx =1$.
Our minimization problem can be stated as follows: given $\alpha>0$ and $c\in [0,|\Omega|]$ we look for
\begin{equation} \label{Lam.d}
\Lambda_\Omega(\alpha,c) := \inf\{ \lambda_\Omega (\alpha,E) \colon E\subset \Omega, |E|=c \},
\end{equation}
that is,   to optimize $\lambda_\Omega(\alpha,E)$ with respect to the class of sets $E\subset \Omega$ of fixed volume.
Let $\chi_E$ be its characteristic function on $E$, then \eqref{lambda.d} and \eqref{lambda.n} are related to (see Proposition \ref{propo}) the eigenvalue problem 
\begin{align} \label{eq.d} \tag{$\mathcal{P}_D$}
\begin{cases}
-\Delta_g u +\alpha \chi_E  g(|u|)\frac{u}{|u|} =\lambda g(|u|)\frac{u}{|u|} &\text{ in } \Omega\\
u=0 &\text{ on }\partial \Omega
\end{cases}
\end{align}
and
\begin{align} \label{eq.n} \tag{$\mathcal{P}_N$}
\begin{cases}
-\Delta_g u + g(|u|)\frac{u}{|u|} + \alpha \chi_E  g(|u|)\frac{u}{|u|} =\lambda g(|u|)\frac{u}{|u|} &\text{ in } \Omega\\
g(|\nabla u|) \frac{\nabla u}{|\nabla u|} \cdot \eta =0 &\text{ on } \partial\Omega,
\end{cases}
\end{align}
respectively, where $\eta$ denotes the outer normal.



An \emph{optimal configuration} for the data $(\Omega,\alpha,c)$ is a minimizer $E$ of \eqref{Lam.d}. If $E$ is an optimal configuration and $u$  is a minimizer of $\lambda_\Omega(\alpha,E)$, then $(u,E)$ is an \emph{optimal pair}.\\
In some applications, for certain design purposes, different forms of the cost functional need to be considered. This leads to take into account optimization problems with different boundary conditions. With the same approach we can also deal with the optimization problem related with the Steklov eigenvalues problem, i.e., the minimization problem \eqref{Lam.d} in the case in which $\lam_\Omega(\alpha,E)$ is a quantity related with 
\begin{align} \label{eq.s} \tag{$\mathcal{P}_S$}
\begin{cases}
-\Delta_g u + g(|u|)\frac{u}{|u|} + \alpha \chi_E g(|u|)\frac{u}{|u|} =0 &\text{ in } \Omega\\
g(|\nabla u|) \frac{\nabla u}{|\nabla u|} \cdot \eta = \lam   g(|u|)\frac{u}{|u|} &\text{ on } \partial\Omega,
\end{cases}
\end{align}
that is, it is defined as
\begin{equation} \label{lambda.s} \tag{$\mathcal{S}$}
\lambda_\Omega (\alpha,E)=\inf\left\{ \int_\Omega G(|\nabla u|) +G(|u|)\,dx+\alpha\int_E G(|u|)\,dx \colon u\in \mathcal{W}(\Omega)\right\},
\end{equation}
where  $\mathcal{W}(\Omega)$  is the set of functions $u\in   W^{1,G}(\Omega)$ such that $\int_{\partial\Omega} G(|u|)\,d\mathcal{H}^{n-1} =1$, being $\mathcal{H}^{n-1}$ the $n-1$-dimensional Hausdorff measure.

Analogously, if $E$ is an optimal configuration and $u$   minimize \eqref{Lam.d}, then $(u,E)$ is an \emph{optimal pair}.

Optimization problems as \eqref{Lam.d} have been already studied in the context of eigenvalues of the Laplacian, and that has been our main inspiration for this article. See  \cite{CGK, CGIKO, DPFBR}. We highlight that the lack of homogeneity of the operator carries out several technical problems we have to face along this article.

Our first main result concerns existence of an optimal pair related to the minimization problems \eqref{Lam.d}  and some properties on them.


\begin{theorem} \label{main1}
For any $\alpha>0$ and $c\in [0,|\Omega|]$ there exists an optimal pair $(u,E)$ such that
\begin{itemize}
\item[(a)] $u\in C^{1,\gamma}{(\Omega)}\cap C^\gamma(\overline\Omega)$ for same $\gamma \in (0, 1)$.
\item[(b)] $E$ is a sublevel set of $u$, i.e.,  there is $t\geq 0$ such that  $E=\{u\leq t\}$.
\item[(c)] Every level set $\{u=s\}$, $s\geq 0$ has Lebesgue measure zero, except possibly when $\alpha$ is an eigenvalue (Dirichlet case) or when $1+\alpha$ is an eigenvalue (Neumann case).
\end{itemize}
\end{theorem}

When  in particular $\Omega$ is the unit ball $B_1$,  we obtain existence of a spherically symmetric optimal configuration:

\begin{theorem} \label{main3}
Fix $\alpha>0$ and $c\in (0,|B_1|)$, then there exists an optimal pair   $(u,E)$ such that $u$ and $E$ are spherically symmetric.
\end{theorem}

Once the set $E$ is fixed, it is easy to see that when $\alpha\to +\infty$ the quantity $\lambda_\Omega(\alpha,E)$ converges to   the minimizer of the problem with $E$ as a hole in $\Omega$ (that is, the minimizing function vanishes on $E$), i.e.,
\begin{equation} \label{lam.inf.E}
\lambda_\Omega(\infty,E)  := \lim_{\alpha\to\infty}\lambda_\Omega(\alpha,E)= \inf_{v\in \mathcal{W}(\Omega), v|_E\equiv 0 }  \int_\Omega G(|\nabla v|)\,dx  .
\end{equation}
The natural limit optimization problem in this case is
\begin{equation}\label{Lam.inf}
\Lambda_\Omega(\infty,c):=\inf\{\lambda_\Omega(\infty,E) \colon E\subset \Omega, |E|=c\}.
\end{equation}
A natural question is whether the optimal configuration of \eqref{Lam.d} converges to these of \eqref{Lam.inf} when $\alpha\to\infty$. The following result answers positively to that issue.


\begin{theorem} \label{main2}
For any sequence $\alpha_k\to\infty$ and optimal pairs $(u_k,E_k)$ of \eqref{Lam.d} there exist a subsequence, still denoted $\alpha_k$, and an optimal pair $(u,E)$  of \eqref{Lam.inf} such that
\begin{align*}
\lim_{k\to\infty} \chi_{E_k}=\chi_E &\quad \text{weakly* in }L^\infty(\Omega),\\
\lim_{k\to\infty} u_k=u &\quad \text{strongly in }\mathcal{W}(\Omega).
\end{align*}
Moreover, $u>0$ in $\Omega\setminus E$. 

\end{theorem}

The paper is organized as follows. In Section \ref{preliminar} we introduce the notation and basic facts on Orlicz-Sobolev spaces used along the manuscript.  Section \ref{sec.autoval} is devoted to study the link between the minimization problems and the eigenvalue problems. In Section  \ref{sec.main1}  we prove our  main results stated in Theorems \ref{main1}, \ref{main3}. Finally, in Section \ref{sec.main2} we prove Theorem \ref{main2}.

\section{Preliminaries}\label{preliminar}
In this section we introduce some notation and basic results on Orlicz-Sobolev spaces that we will use in this paper.
\subsection{Young functions}
An application $G\colon\R_+\to \R_+$ is said to be a  \emph{Young function} if it admits the integral formulation $G(t)=\int_0^t g(\tau)\,d\tau$, where the right continuous function $g$ defined on $[0,\infty)$ has the following properties:
\begin{align*}
&g(0)=0, \quad g(t)>0 \text{ for } t>0 \label{g0} \tag{$g_1$}, \\
&g \text{ is non-decreasing on } (0,\infty) \label{g2} \tag{$g_2$}, \\
&\lim_{t\to\infty}g(t)=\infty  \label{g3} \tag{$g_3$} .
\end{align*}
From these properties it is easy to see that a Young function $G$ is continuous, non negative, strictly increasing and convex on $[0,\infty)$.

We will assume the following growth condition on Young functions: there exist fixed constants $p^\pm$ such that
$$
1\leq p^- \leq \frac{tg(t)}{G(t)} \leq p^+ <\infty, \qquad \text{ for all } t>0.
$$
For technical reasons, we will also assume that $g(t)$ is convex   for all $t\geq 0$.

The following  properties on Young functions are well-known. See for instance \cite{FJK} for the proof of these results.

\begin{lemma} \label{lema.prop}
Let $G$ be a Young function satisfying \eqref{cond} and $a,b\geq 0$. Then
\begin{align*}
  &\min\{ a^{p^-}, a^{p^+}\} G(b) \leq G(ab)\leq   \max\{a^{p^-},a^{p^+}\} G(b),\tag{$L_1$}\label{L1}\\
  &G(a+b)\leq \mathbf{C} (G(a)+G(b)) \quad \text{with } \mathbf{C}:=  2^{p^+},\tag{$L_2$}\label{L2}\\
	&G \text{ is Lipschitz continuous}. \tag{$L_3$}\label{L_3}
 \end{align*}
\end{lemma}
Condition \eqref{cond} is known as the \emph{$\Delta_2$ condition} or \emph{doubling condition} and, as it is showed in \cite[Theorem 3.4.4]{FJK}, it is equivalent to the right hand side inequality in \eqref{cond}.

The \emph{complementary} Young function of a Young function $G$ is defined as
$$
\displaystyle \tilde{G}(t): = \sup_{s\geq0}\{st-G(s)\}.
$$
It is easy to see that the left hand side inequality in \eqref{cond} is equivalent to assume that $\tilde G$ satisfies the $\Delta_2$ condition.\\
The following useful Lemma can be found in \cite[Lemma 2.9]{BoS}.
\begin{lemma}\label{lemita}
Let $G$ be a Young function and let $\tilde G$ be its complementary function. Then
$$
\tilde G(g(t)) \leq C G(t).
$$
\end{lemma}
Finally, we state the following version of the so-called Bathtub principle in our settings. See for instance \cite[Theorem 1.14]{LL}.
\begin{prop} \label{bathtube}
Then, the minimization problem 
$$
 \inf\left\{ \int_\Omega \eta G(|u|)\,dx\colon  0\leq \eta \leq 1, \int_\Omega \eta \,dx =c \right\}
$$
has a solution  $\eta= \chi_{E}(x)$ where $E$ is any set with $|E|=c$ and
$$
\{u<t\} \subset E \subset \{u\leq t\}, \quad t=\sup\{s \colon |\{u<s\}|<c\}.
$$
\end{prop}
\subsection{Orlicz-Sobolev spaces}
Given a   Young function  $G$ and a bounded set $\Omega$ we consider the spaces
\begin{align*}
	L^G(\Omega) & :=\{ u : \R\to \R \text{ measurable such that }\int_\Omega G(|u|)\,dx<\infty\},
	\\L^G(\partial\Omega) & :=\{ u : \R\to \R \text{ measurable such that } \int_{\partial\Omega} G(|u|)\, d\mathcal{H}^{n-1}<\infty\},
	\\W^{1,G}(\Omega)& :=\{ u \in L^{G}(\Omega) \text{ such that }
\int_\Omega G(|\nabla u|)\,dx <\infty\}.
\end{align*}
These spaces are endowed with the so called  \textit{Luxemburg norm} defined as follows
\begin{align*}
\|u\|_{L^G(\Omega)}&=\inf \left\{ \lambda>0 : \int_\Omega G \left(
\frac{u}{\lambda}\right)\,dx \leq 1 \right\},\\
\|u\|_{L^G(\partial\Omega)}&=\inf \left\{ \lambda>0 : \int_{\partial\Omega} G \left(
\frac{u}{\lambda}\right)\,d\mathcal{H}^{n-1}\leq 1 \right\},\\
\|u\|_{W^{1,G}(\Omega)} &= \|u\|_{L^G(\Omega)} + \|\nabla
u\|_{L^G(\Omega)},
\end{align*}
and  are reflexive and separable Banach spaces if and only if   $G$ y $\tilde{G}$ satisfies the    $\Delta_2$ condition.

We define  $W^{1,G}_0(\Omega)$ as the subset of functions in $W^{1,G}(\Omega)$ such that $u=0$ in $\partial\Omega$.
We conclude this subsection by recalling that under our assumptions, boundedness of the modular implies boundedness of the norm (see for instance \cite[Lemma 2.1.12]{DHHR}).
\begin{lemma} \label{acotado}
Let $G$ be a Young function satisfying \eqref{cond}. Then, if $\int_{\partial\Omega}G(|u|)\,d\mathcal{H}^{n-1}$ (resp. $\int_\Omega G(|u|)\,dx$) is bounded, then $\|u\|_{L^G(\partial \Omega)}$ (resp. $\|u\|_{L^G(\Omega)}$) is bounded.
\end{lemma}


We also recall the following Poincar\'e's inequality: given a bounded set $\Omega\subset \R^n$, for any $u\in W^{1,G}(\Omega)$ such that $|\{x\in \Omega \colon u(x)=0\}|=\kappa >0$ it holds that 
$$
\int_\Omega G(|u|)\,dx \leq C \int_\Omega G(|\nabla u|)\,dx 
$$
for some positive constant $C$ depending only on $n$, $\kappa$ and $p^\pm$. The same conclusion holds when $u\in W^{1,G}_0(\Omega)$. With this inequality together with  Lemma \ref{acotado}, we can ensure that the limit problem \eqref{lam.inf.E} is well possed.

\subsection{Some embedding results} \label{sec.emb}
In order to guarantee compact embeddings  the following conditions will be assumed
\begin{equation} \label{cond1}
\int_{0}^{1}\frac{G^{-1}(s)}{s^{1+\frac{1}{n}}}ds< \infty  \quad \text{and}\quad  \int_{1}^{\infty }\frac{G^{-1}(s)}{s^{1+\frac{1}{n}}}ds=\infty.
\end{equation}
For any Young function satisfying \eqref{cond1},  the \emph{Sobolev critical function} is defined as
$$
G_{*}^{-1}(t)=\int_{0}^{t}\frac{G^{-1}(s)}{s^{\frac{n+1}{n}}}ds.
$$

Combining  \cite[Theorems 7.4.6 and 7.4.6]{FJK} with \cite[Example 6.3]{DSSSS} the following embedding  holds.

\begin{prop}\label{compact}
	Let $G$ be a Young function satisfying \eqref{cond} and \eqref{cond1}. Let $\Omega\subset \R^n$ be a $C^{0,1}$ bounded open subset. Then   the embeddings
	$$
	W^{1,G}(\Omega)\hookrightarrow L^{G}(\Omega), \qquad
	W^{1,G}(\Omega)\hookrightarrow L^{G}(\partial\Omega)
	$$
	are compact.
\end{prop}

The following vectorial inequality will be of use for our purposes. See Lemma 3.1 in \cite{CSS}.
\begin{lemma} \label{mono}
Let $G$ be a Young function satisfying that $G'=g$ is a convex function. Then there exists $C=C(p^-)$ such that
$$
\left( g(|a|)\frac{a}{|a|} - g(|b|)\frac{b}{|b|}\right)\cdot (a-b) \geq C G(|a-b|).
$$
\end{lemma}
 
 
\section{The minimization problems and their related eigenvalue problems} \label{sec.autoval}

From now on, we will always assume that $G$ is a Young function  satisfying \eqref{cond} and \eqref{cond1}. Along this section $\alpha>0$ and $E\subset \Omega$ are fixed.


In this section  we state some properties on   $\lam_\Omega(\alpha,E)$. For that end, we introduce the functionals $\mathcal{I},\mathcal{J},\mathcal{K}\colon W^{1,G}(\Omega) \to \R$  are defined as follows
\begin{align*}
\mathcal{I}(v)&:=\int_\Omega G(|\nabla v|) + G(|u|)\,dx + \alpha\int_E G(|v|)\,dx\\
\mathcal{J}(v)&:=\int_\Omega G(|v|)\,dx, \qquad \mathcal{K}(v):=\int_{\partial \Omega} G(|v|)\,d\mathcal{H}^{n-1}
\end{align*}
and whose Frech\'et derivatives are defined from $W^{1,G}(\Omega)$ (in the Neumann or Steklov case, or $W^{1,G}_0(\Omega)$ in the Dirichlet case) into its dual space, and they are given by
\begin{align*}
\langle \mathcal{I}'(u),v \rangle &=\int_\Omega g(|\nabla u|)\tfrac{\nabla u}{|\nabla u|} \cdot \nabla v +  g(|u|) \tfrac{u}{|u|}v\,dx +\int_E \alpha g(|u|) \tfrac{u}{|u|}v\,dx\\
\langle \mathcal{J}'(u),v\rangle&=\int_{\Omega} g(|u|)\tfrac{u}{|u|}v\,dx, \qquad \langle \mathcal{K}'(u),v\rangle=\int_{\partial \Omega} g(|u|)\tfrac{u}{|u|}v\,d\mathcal{H}^{n-1}.
\end{align*}

We say that $\lam$ is an \emph{eigenvalue} of \eqref{eq.d}  (resp. of \eqref{eq.n}) with associated \emph{eigenfunction} $u\in W^{1,G}_0(\Omega)$ (resp. $u\in W^{1,G}(\Omega)$) if
$$
\langle \mathcal{I}'(u),v \rangle  = \lam \langle \mathcal{J}'(u),v \rangle 
$$
for any $v\in W^{1,G}_0(\Omega)$ (resp. $v\in W^{1,G}(\Omega)$).

In the Dirichlet we can neglect the term $\int_\Omega G(|u|)\,dx$ in $\mathcal{I}(u)$ due to the Poincar\'e's inequality.

Similarly, $\lam$ is an eigenvalue of the Steklov equation \eqref{eq.s} with eigenfunction $u\in W^{1,G}(\Omega)$ if
$$
\langle \mathcal{I}'(u),v \rangle  = \lam \langle \mathcal{K}'(u),v \rangle 
$$
for any $v\in W^{1,G}(\Omega)$.

We prove now that the quantities defined in \eqref{lambda.d}, \eqref{lambda.n} and \eqref{lambda.s} are comparable with eigenvalues of \eqref{eq.d}, \eqref{eq.n} and \eqref{eq.s}, and  eigenfunctions of   these eigenproblems minimize \eqref{lambda.d}, \eqref{lambda.n} and \eqref{lambda.s}, respectively.

Observe that $\lam_\Omega(\alpha,E)$ defined in \eqref{lambda.d}, \eqref{lambda.n} and  \eqref{lambda.n} can be written  as
\begin{equation}\label{lambdaD}  
\inf\left\{\frac{\mathcal{I}(v)}{\mathcal{J}(v)}\colon v\in W^{1,G}_0(\Omega), \int_\Omega G(|v|)\,dx=1\right\} \qquad \text{(Dirichlet)}
\end{equation}
\begin{equation}\label{lambdaN}  
\inf\left\{\frac{\mathcal{I}(v)}{\mathcal{J}(v)}\colon v\in W^{1,G}(\Omega), \int_\Omega G(|v|)\,dx=1\right\} \qquad \text{(Neumann)}
\end{equation}
\begin{equation}\label{lambdaS}  
\inf\left\{\frac{\mathcal{I}(v)}{\mathcal{K}(v)}\colon v\in W^{1,G}(\Omega), \int_{\partial\Omega} G(|v|)\,d\mathcal{H}^{n-1}=1\right\} \qquad \text{(Steklov)}.
\end{equation}

The existence and continuity of such  derivatives is guaranteed since $G$ and $\tilde G$ satisfy the $\Delta_2$ condition.

We start by proving that  $\lam_\Omega(\alpha,E)$ is attained by a suitable function.
\begin{prop} \label{prop1}
There exists $u_D\in W^{1,G}_0(\Omega)$ solving \eqref{lambdaD} such that $\int_\Omega G(|u_D|)\,dx=1$. Similarly, there is $u_N\in W^{1,G}(\Omega)$ solving \eqref{lambdaN} such that $\int_\Omega G(|u_N|)\,dx=1$ and $u_S\in W^{1,G}(\Omega)$ solving \eqref{lambdaS} such that $\int_\Omega G(|u_S|)\,d\mathcal{H}^{n-1}=1$.
\end{prop}

\begin{proof}
We start with the Dirichlet case. Let $\{u_k\}_{k \in \N}$ be a minimizing sequence for $\lam_\Omega(\alpha,E)$, that is
\begin{equation} \label{minim}
u_k\in W^{1,G}_0(\Omega), \quad \mathcal{J}(u_k)=1 \quad \text{ and }\quad  \lim_{k\to\infty} \mathcal{I}(u_k) = \lambda_\Omega(\alpha,E).
\end{equation}
Observe that $\{u_k\}_{k\in \N}$ is bounded in $L^G(\Omega)$ and for $k$ big enough, $\int_\Omega G(|\nabla u_k|)\,dx \leq \mathcal{I}(u_k) \leq 1 + \lambda(\alpha,E)$. Therefore, from the reflexivity of $W^{1,G}(\Omega)$ together with the compact embedding stated in Proposition \ref{compact}  we get that there exists a function $u\in W_0^{1,G}(\Omega)$ such that, up to a subsequence
\begin{align*}
&u_k\rightharpoonup u\; \text{ weakly in } W^{1,G}(\Omega),\\
&u_k\to u\;  \text{ strongly  and a.e. in } L^{G}(\Omega).
\end{align*}
The strong convergence in $L^G(\Omega)$ implies that $\mathcal{J}(u)=1$. Hence, by definition
$$
\lambda_\Omega(\alpha,E) \leq \mathcal{I}(u).
$$
On the other hand, since modulars are lower semi-continuous with respect to the weak convergence, by Fatou's Lemma we get
\begin{align*}
\mathcal{I}(u) \leq \liminf_{k\to\infty}\mathcal{I}(u_k)=\lambda_\Omega(\alpha,E).
\end{align*}
From the last two inequalities the lemma follows in the Dirichlet case. The proof for the  remaining cases is analogous.
\end{proof}

We remark that the numbers defined in \eqref{lambdaD}, \eqref{lambdaN} and \eqref{lambdaS} are attained by an eigenfunction corresponding to an eigenvalues to \eqref{eq.d}, \eqref{eq.n} and \eqref{eq.s}, respectively:

\begin{prop} \label{propo}
The minimization problem  $\lam_\Omega(\alpha,E)$ defined in \eqref{lambdaD} is attained by an eigenfunction corresponding to an eigenvalue $\ell_\Omega(\alpha,E)$ of  \eqref{eq.d} with Dirichlet boundary condition. Moreover, 
\begin{equation} \label{relacion}
\frac{p^-}{p^+} \lam_\Omega(\alpha,E) \leq  \ell_\Omega(\alpha,E) \leq \frac{p^+}{p^-} \lam_\Omega(\alpha,E).
\end{equation}
A similar conclusion holds for \eqref{lambdaN} and \eqref{lambdaS}.
\end{prop}

\begin{proof}
We deal with the Dirichlet case, the remaining cases are similar. By Proposition \ref{prop1} there is $u \in W^{1,G}_0(\Omega)$ such that $\mathcal{J}(u)=1$ and $\lam_\Omega(\alpha,E) =\frac{\mathcal{I}(u)}{\mathcal{J}(u)}$. 

By Lagrange's multiplier rule, there exists $c_\Omega(\alpha,E)$ such that
\begin{equation*}
\langle \mathcal{I}'(u),v \rangle = \ell_\Omega(\alpha,E) \langle  \mathcal{J}'(u),v \rangle
\end{equation*}
for all $v \in W_0^{1,G}(\Omega)$, that is, $\ell_\Omega(\alpha,E)$ is an eigenvalue of \eqref{eq.d} with eigenfunction $u$.  Moreover, since for any $t>0$
$$
p^- G(t) \leq g(t)t \leq p^+ G(t)
$$
we get that
$$
\frac{p^-}{p^+} \frac{\mathcal{I}(u)}{\mathcal{J}(u)} \leq \frac{\langle \mathcal{I}'(u),u \rangle}{ \langle  \mathcal{J}'(u),u \rangle}  \leq \frac{p^+}{p^-} \frac{\mathcal{I}(u)}{\mathcal{J}(u)} 
$$
which means \eqref{relacion}.

This concludes the proposition.
\end{proof}

Eigenfunctions are positive in $\Omega$ and H\"older continuous up to the boundary:

\begin{prop} \label{propo}
Let $u\in \mathcal{W}$ solving $\lam_\Omega(\alpha,E)$. Then $u\in C^{1,\gamma}(\Omega)\cap C^\gamma(\bar \Omega)$ and $u>0$ in $\Omega$.
\end{prop}

\begin{proof}
Let $u\in \mathcal{W}$ solving $\lam_\Omega(\alpha,E)$ either in the Dirichlet, Neumann or Steklov case.  Observe that both $u$ and $|u|$ solve the  \eqref{lambdaD}, \eqref{lambdaN} and \eqref{lambdaS}, so, we can assume $u\geq 0$ in $\Omega$. Moreover, by the strong maximum principle (see \cite{Monte})  we have that $u>0$ in $\Omega$. The  regularity estimates of  $u$ follow from \cite{L,Li}.
\end{proof}

\section{An existence result} \label{sec.main1}

This section is devoted to prove the existence of an optimal configuration of \eqref{Lam.d} and to analyze some properties which it fulfills, namely, Theorem \ref{main1}.

\begin{proof}[Proof of Theorem \ref{main1}]
Let us start with the Dirichlet case. Let $\{E_k\}_{k\in\N}$ be a minimizing sequence for $\Lambda:=\Lambda_\Omega(\alpha,c)$, that is, $\lam(E_k):=\lam_\Omega(\alpha,E_k) \to \Lambda$ as $k\to\infty$. Let $u_k\in W^{1,G}_0(\Omega)$ be   a minimizer  to $\lam(E_k)$ normalized such that $\mathcal{J}(u_k)=1$. For $k$ big enough  we have that
\begin{equation} \label{ec.k}
\int_\Omega G(|\nabla  u_k|)+ \alpha \int_{E_k} G(|u_k|)\,dx =\lam(E_k)\leq 1+\Lambda,
\end{equation}
from where it follows that $\{u_k\}_{k\in\N}$ is bounded in $W^{1,G}_0(\Omega)$. Moreover, the sequence $\{\chi_{E_k}\}_{k\in\N}$ is bounded in $L^G(\Omega)$. Then, there exist $u\in W^{1,G}_0(\Omega)$ and $\eta\in L^G(\Omega)$ such that, up to a subsequence,	
\begin{align*}
&u_{k}\rightharpoonup u\; \text{ weakly in } W^{1,G}_0(\Omega),\\
&u_{k}\to u\;  \text{ strongly and a.e. in } L^{G}(\Omega),\\
&\chi_{E_k}\rightharpoonup \eta\;  \text{ weakly in } L^{G}(\Omega).
\end{align*}
From this, $\int_\Omega \eta \,dx =c$ and  $\int_\Omega G(|u|)\,dx=1$. 

Consider the eigenvalue $\ell(E_k)$ given in Proposition \ref{propo}, which has eigenfunction $u_k$. Then
\begin{equation} \label{weak.d}
\int_\Omega g(|\nabla u_k|)\tfrac{\nabla u_k}{|\nabla u_k|} \cdot \nabla v \,dx  =  \int_\Omega (\ell(E_k)-\alpha \chi_{E_k} ) g(|u_k|) \tfrac{u_k}{|u_k|}v\,dx \qquad \forall v\in W^{1,G}_0(\Omega).
\end{equation} 
 
We test with $u_k-u$ in \eqref{weak.d}
\begin{align*} 
\int_\Omega& \left( g(|\nabla u_k|)\tfrac{\nabla u_k}{|\nabla u_k|} - g(|\nabla u|)\tfrac{\nabla u}{|\nabla u|} \right)\cdot (\nabla u_k-\nabla u)  \,dx  =\\
&=  \int_\Omega (\ell(E_k)-\alpha \chi_{E_k} ) g(|u_k|) \tfrac{u_k}{|u_k|} (u_k-u)\,dx - 
\int_\Omega g(|\nabla u|)\tfrac{\nabla u}{|\nabla u|} \cdot (\nabla u_k-\nabla u) \,dx\\
&:=(I)+(II).
\end{align*} 
Observe that $(II)$ goes to $0$ when $k\to\infty$ due to the weak convergence of the gradients; moreover, using Proposition \ref{propo}, Lemma \ref{lemita} and Holder's inequality for Orlicz functions we get 
$$
(I)\leq \frac{p^+}{p^-}\Lambda \| g(u_k)\|_{L^{\tilde G}(\Omega)} \|u_k-u\|_{L^G(\Omega)} \to 0 \quad \text{ as } k\to\infty
$$
since $u_k\to u$ strongly in $L^G(\Omega)$. Therefore, 
$$
\lim_{k\to\infty} \int_\Omega  \left( g(|\nabla u_k|)\tfrac{\nabla u_k}{|\nabla u_k|} - g(|\nabla u|)\tfrac{\nabla u}{|\nabla u|} \right)\cdot (\nabla u_k-\nabla u)  \,dx =0.
$$
From this, by using the vectorial inequality given in Lemma \ref{mono} we get that
$$
\lim_{k\to\infty} \int_\Omega G(|\nabla (u_k -u)|)\,dx=0
$$
and therefore the strong convergence of the gradients in $L^G(\Omega)$:
\begin{equation*}
\lim_{k\to\infty}\int_\Omega G(|\nabla u_k|) \,dx = \int_\Omega G(|\nabla u|)\,dx.
\end{equation*}



Moreover, using that $u_k\to u$ strongly in $L^{G}(\Omega)$ and $\chi_{E_k}\rightharpoonup \eta$ weakly in $L^G(\Omega)$, we get
$$
\int_\Omega \chi_{E_k} G(|u_k|)  \,dx \to \int_\Omega \eta G(|u|) \,dx.
$$

Also, since $0\leq \chi_{E_k}\leq 1$ for all $k$, and weak convergence preserves pointwise inequalities, we have $0\leq \eta\leq 1$ a.e.


Therefore, taking limit as $k\to\infty$ in \eqref{ec.k}, we get
\begin{equation}\label{Lambda}
\int_\Omega G(|\nabla u|) + \alpha \eta G(|u|) \,dx = \Lambda.
\end{equation}

In order to conclude with the proof, let us characterize $\eta$. By Proposition \ref{bathtube}, the problem
\begin{equation}\label{set}
\inf\left\{\int_\Omega \eta G(|u|)\,dx\colon  0\leq\eta\leq 1, \int_\Omega \eta\,dx=c\right\}
\end{equation}
has a solution $\eta= \chi_{E}(x)$, where $E$ is any set with $|E|=c$ and
$$
\{u<t\} \subset E \subset \{u\leq t\}, \quad t=\sup\{s \colon |\{u<s\}|<c\}.
$$
Therefore,
\begin{equation}\label{cara}
\int_\Omega \chi_E G(|u|)\,dx\leq \int_\Omega \eta G(|u|)\,dx.
\end{equation}
Combining \eqref{Lambda} and \eqref{cara}, we obtain that
$$
\int_\Omega G(|\nabla u|) + \alpha \chi_E G(|u|) \,dx \leq \Lambda.
$$
On the other hand, by definition \eqref{Lam.d} of $\Lambda$ we get
$$
\Lambda \leq \lam(E)= \int_\Omega G(|\nabla u|) + \alpha \chi_E G(|u|) \,dx.
$$
This concludes the proof of the existence of an optimal pair in the Dirichlet case.

The proof in the Neumann case is completely analogous just taking the test functions in $W^{1,G}(\Omega)$ instead of $W^{1,G}_0(\Omega)$. For the Steklov problem the proof is similar by using that Proposition \ref{compact} additionally gives that $u_k\to u$ strongly in $L^G(\partial \Omega)$ due to the compact embedding, which allows to conclude that $\mathcal{K}(u)=1$.


{\bf Proof of (a)}. It follows from Proposition \ref{propo}.

{\bf Proof of (b) and (c)}.

Let us consider the Dirichlet case. Let $(u,E)$ be an optimal pair. We claim that 
$$
\{u<t\} \subset E \subset \{u\leq t\}, \quad t=\sup\{s \colon |\{u<s\}|<c\}
$$
always hold up to a set of measure zero. If $E$ does not satisfy the previous condition, then we can reduce $\int_E G(|u|)\,dx$ by shifting a part of $D$ from $\{u>t\}$ to $\{u\leq t\}$.

Finally, set $\mathcal{R}_s=\left \{x\in\Omega\colon u(x)=s \right \}$ for any $s>0$. Since,   $\nabla u=0$ a.e. on $\mathcal{R}_s$, we get  $\Delta_g u=0$ a.e. on $\mathcal{R}_s$. Then, since $u$ is eigenfunction with some eigenvalue $\ell_\Omega$
$$
(\ell_\Omega(\alpha,E)-\alpha\chi_E)g(|u|)  =0 \quad \text{ a.e. on } \mathcal{R}_s.
$$
Since $(\ell_\Omega(\alpha,E)-\alpha\chi_E)g(|u|)> 0$ a.e. on  $\mathcal{R}_s,$ (except possibly when $\ell_\Omega(\alpha,E)=\alpha$), it follows that $|\mathcal{R}_s|=0$. This proves $(c)$ in the Dirichlet case. The Neumann case is totally analogous, and the same conclusion holds except possibly when $\ell_\Omega(\alpha,E)=1+\alpha$. For the Steklov problem, with the same reasoning we get 
$$
(1+\alpha\chi_E)g(|u|)  =0 \quad \text{ a.e. on } \mathcal{R}_s,
$$
from where $|\mathcal{R}_s|=0$. 
In particular, when $s=t$ we get the last assertion in $(b)$.
\end{proof}


In the Dirichlet and Neumann case, the following result establishes the dependence of the parameter dependence of $\Lambda$.
\begin{theorem}
The application $(\alpha,c)\mapsto \Lambda(\alpha,c)$ is Lipschitz continuous, uniformly on bounded sets: for any $0\leq \alpha,\alpha'$ and $0\leq c, c' \leq |\Omega|$ then
$$
|\Lambda(\alpha,c)-\Lambda(\alpha',c')|\leq  |\alpha'-\alpha| + |c'-c| \min\{\alpha,\alpha'\} C, 
$$
where $C$ is a constant depending of $p^+$ and $p^-$.
\end{theorem}

\begin{proof}
Let us consider the Dirchlet case, the Neumann case follows analogously.
By symmetry assume that $c'\geq c$. Denote $\Lambda=\Lambda(\alpha,c)$ and $\Lambda'=\Lambda(\alpha',c')$. Let $(u,E)$ and $(u',E')$ be minimizers of $\Lambda$ and $\Lambda'$, respectively, normalized such that $\int_\Omega G(|u|)\,dx=\int_\Omega G(|u'|)\,dx=1$.

Since $|E|=c\leq c'=|E'|$ we can choose sets $E_1$ and $E_1'$ such that $E_1\subset E'$ and $E\subset E_1'$ with $|E_1|=c$ and $|E_1'|=c'$. We can also assume that $E_1'$ is a sublevel set of the form $\{u\leq s\}$ for a suitable $s$. Since $(u,E)$ is an optimal pair and using that $E_1$ is admissible in the minimization problem for $\Lambda$, and $u'$ is admissible in the characterization of $\lam_\Omega(\alpha,E_1)$ we get 
\begin{align*}
\Lambda &\leq \lambda_\Omega(\alpha,E_1) \leq \int_\Omega G(|\nabla u'|)+ \alpha\int_{E_1} G(|u'|)\,dx\\
&=\Lambda' + (\alpha-\alpha')\int_{E'} G(|  u'|)\,dx-\alpha\int_{E'\setminus E_1} G(|u'|)\,dx.
\end{align*}
With a similar argument we also get
\begin{align*}
\Lambda' &\leq \lambda_\Omega(\alpha',E_1') \leq \int_\Omega G(|\nabla u|) +\alpha'\int_{E_1'} G(|u|)\,dx\\
&=\Lambda + (\alpha'-\alpha)\int_{E'_1} G(|  u|)\,dx + \alpha \int_{E_1'\setminus E} G(|u|)\,dx\\
&=\Lambda + (\alpha'-\alpha)\int_{E} G(|  u|)\,dx + \alpha'\int_{E_1'\setminus E} G(|u|)\,dx.
\end{align*}

From these inequalities we obtain that
\begin{align*}
|\Lambda-\Lambda'|&\leq |\alpha'-\alpha| \max\left\{ \int_{E'} G(|  u'|)\,dx, \int_{E} G(|  u|)\,dx  \right\}\\& +\max\{\alpha,\alpha'\} \max\left\{\int_{E_1'\setminus E} G(|u|)\,dx, \;  \int_{E'\setminus E_1} G(|u'|)\,dx  \right\}\\
&:=(I)+(II).
\end{align*}

In order to bound  observe that
$$
\int_E G(|u|)\,dx \leq \int_\Omega G(|u|)\,dx =1, \qquad \int_{E'} G(|u'|)\,dx \leq \int_\Omega G(|u'|)\,dx =1, 
$$
from where
$$
(I)\leq |\alpha'-\alpha|.
$$

To deal with the second integrals, we notice that
$$
\int_{E_1'\setminus E} G(|u|)\,dx \leq |E_1'\setminus E| G(\|u\|_{L^\infty(\Omega)}), \quad
\int_{E'\setminus E_1} G(|u'|)\,dx \leq |E'\setminus E_1| G(\|u'\|_{L^\infty(\Omega)}).
$$
We observe that both $\Lambda$ and $\Lambda'$ can be bounded uniformly with a constant independent of $\alpha$ and $E$ (resp. $\alpha'$ and $E'$) due to \eqref{cota.indep}. Therefore, by Lemma 3.1 of \cite{L83}, $u$ (in the Dirchlet or Neumann setting) can be bounded as
$$
\|u\|_{L^\infty(\Omega)}\leq C(p^+,p^-, \|u\|_{L^1(\Omega)}), \qquad \|u'\|_{L^\infty(\Omega)}\leq C(p^+,p^-, \|u'\|_{L^1(\Omega)}),
$$
and since $\int_\Omega G(|u|)\,dx=\int_\Omega G(|u'|)\,dx=1$, we get that
$$
\int_{E_1'\setminus E} G(|u|)\,dx \leq (c'-c)  C(p^+,p^-), \quad
\int_{E'\setminus E_1} G(|u'|)\,dx \leq (c'-c) C(p^+,p^-).
$$
Therefore, 
$$
(II)\leq |c'-c|C(p^+,p^-) \max\{\alpha',\alpha\}.
$$
This concludes the proof.

\end{proof}

 
We conclude this section by proving Theorem \ref{main3}.  For that end, we recall some basic results on \emph{spherical symmetrization} of functions on Orlicz-Sobolev spaces.

Given a mensurable set $\Omega\subset\R^n$, the spherical symmetrization $\Omega^*$ of $\Omega$ with respect to an axis given by a unit vector $e_k$ reads as follows: for each positive number $r$, take the intersection $\Omega\cap \partial B(0,r)$ and replace it by the spherical cap of the same $\mathcal{H}^{n-1}-$measure and center $re_k$. Hence, $\Omega^{*}$ is the union of these caps.

Now, the spherical symmetrization $u^*$ of a measurable function $u\colon \Omega \to \R_{+}$ is constructed by symmetrizing the super-level sets so that, for all $t$
$$
   \{u^*\geq t\} = \{u\geq t\}^*.
$$
We refer to \cite{Kawohl}  for more details.

Denoting  by $B_1$ the ball of unit radius centered at the origin, the following properties on symmetrization can be found in \cite[Theorem 4.1]{DSSSS} (see also \cite{Kawohl}).

 \begin{prop}\label{reluu*} Let $u \in W^{1, G}(B_1)$ and $u^{*}$ be its spherical symmetrization.  Then, $u^* \in W^{1, G}(B_1)$. Moreover,
\begin{itemize}
  \item[(i)] $\displaystyle \int_{B_1}G(|u^*|)\,dx =  \int_{B_1}G(|u|)\,dx$,
  \item[(ii)] $\displaystyle \int_{\partial B_1}G(|u^{*}|)\,d\mathcal{H}^{n-1} =  \int_{\partial B_1}G(|u|)\,d\mathcal{H}^{n-1}$,
  \item[(iii)] $\displaystyle \int_{B_1}G(|\nabla u^{*}|)\,dx \leq   \int_{B_1}G(|\nabla u|)\,dx$,
  \item[(iv)] $\displaystyle \int_{B_1} (\alpha \chi_D)_* G(|u^*|)\,dx \leq \int_{B_1}\chi_{D} G(|u|)\,dx$,
\end{itemize}
where $D\subset B_1$ and $(\alpha\chi_D)_* = -(-\alpha\chi_D)^*$.
\end{prop}

We are now in position to prove our symmetrization result for optimal pairs.

\begin{proof}[Proof of Theorem \ref{main3}]
Given  $\alpha>0$ and $c\in (0,|B_1|)$, let $(u,E)$ be an optimal pair and let $u^*$ be the spherical symmetrization of $u$. Consider the set $E^*$ defined by $\chi_{E^*}=(\chi_E)_*$.

In the Dirichlet or Neumann case, from item (i) in Proposition \ref{reluu*}, $\mathcal{J}(u^*)=1$ since $\mathcal{J}(u)=1$. Then, using Proposition \ref{reluu*} again we get
$$
\lambda_\Omega(\alpha,E^*) \leq \mathcal{I}(u^*) \leq \mathcal{I}(u) = \lambda_\Omega(\alpha,E).
$$
The same assertion holds in the Steklov case since from item (ii) in  Proposition \ref{reluu*}, $\mathcal{K}(u^*)=1$.

Finally, since $|E^*|=c=|E|$, $(u^*,E^*)$ is also an optimal pair.
\end{proof}

\section{Limit as $\alpha\to\infty$} \label{sec.main2}

Fixed a subset $E\subset\Omega\subset \R^n$ and $\alpha\in \R$, in our first result we have focused on the study of the optimization problem
$$
\Lambda_\Omega(\alpha,c):= \inf\{  \lambda_\Omega(\alpha,E)\colon E\subset \Omega, |E|=c \}
$$
where $\lambda_\Omega(\alpha,E)$ is related to the Dirichlet, Neumann or Steklov eigenvalue problem.

In this section, fixed the value of $c$ we analyze the behavior of $\Lambda_\Omega(\alpha,c)$ as $\alpha\to\infty$. We recall that the corresponding limit problem is defined as
\begin{align*}
\Lambda_\Omega(\infty,c):=\inf\left\{  \lambda_\Omega(\infty,E)\colon E\subset \Omega, |E|=c \right\}
\end{align*}
where
$$
\lam_\Omega(\infty,E)=\inf_{v\in \mathcal{W}(\Omega), v|_E\equiv 0 } \int_\Omega G(|\nabla v|)\,dx  .
$$

Now we are in position to prove our second main result.

\begin{proof}[Proof of Theorem \ref{main2}]
We start with the proof in the Dirichlet case. Let $\alpha,c >0$ be fixed and let $(u_\alpha,E_\alpha)$ be an optimal pair related to
$$
\Lambda_\Omega(\alpha,c)=\inf\{\lambda_\Omega(\alpha,E)\colon E\subset\Omega, |E|=c\}, 
$$
i.e., $E_\alpha$ solves $\Lambda_\Omega(\alpha,c)$ and $u_\alpha$ is the minimizer corresponding to $\lam_\Omega(\alpha,E)$.

Let $u_0 \in W^{1,G}_0(\Omega)$ and $E_0\subset \Omega$ be such that $\left| E_0 \right|=c$ and
$u_0 \chi_{E_0}=0$.
Then, we have that
\begin{align} \label{cota.indep}
\begin{split}
\Lambda_\Omega(\alpha,c)&\leq 
\frac{\int_\Omega G(|\nabla u_0|)\,dx + \alpha\int_{E_0} G(|u_0|)\,dx}{\int_{\Omega}G(|u_0|)\,dx}\\
&=
\frac{\int_\Omega G(|\nabla u_0|)\,dx }{\int_{\Omega}G(|u_0|)\,dx} :=C
\end{split}
\end{align}
with $C$ independent of $\alpha$. Thus ${\Lambda_\Omega(\alpha,c)}$ is a bounded sequence in $\R$. Moreover, since for each fixed $E$, $\lambda_\Omega(\alpha,E)$ is increasing in $\alpha$, then it is clear that ${\Lambda_\Omega(\alpha,c)}$ is also increasing in $\alpha$.

Consequently, $\{u_\alpha\}_{\alpha>0}$ is bounded in $W^{1,G}(\Omega)$. Moreover, the sequence $\{\chi_{E_\alpha}\}_{\alpha>0}$ is bounded in $L^\infty(\Omega)$. Therefore by Proposition \ref{compact}, up to a subsequence, there exist $u_\infty\in W^{1,G}(\Omega)$ and   $\eta_\infty \in L^\infty (\Omega)$ such that
\begin{align*}
&u_{\alpha_k}\rightharpoonup u_\infty \; \text{ weakly in } W^{1,G}_0(\Omega),\\
&u_{\alpha_k}\to u_\infty \; \text{ a.e. in $\Omega$ and strongly in } L^{G}(\Omega),\\
&\chi_{E_{\alpha_k}}\rightharpoonup \eta_\infty \; \text{ weakly*  in }L^{\infty }(\Omega).
\end{align*} 
Therefore, $\int_\Omega G(|u_{\infty}|)\,dx=1$ and 
$$
\int_\Omega \chi_{E_{\alpha_k}} G(|u_{\alpha_k}|)\,dx \to \int_\Omega \eta_\infty G(|u_\infty|)\,dx,
$$
Moreover, since $0\leq \chi_{E_{\alpha_k}}\leq 1$ for all $k$, and weak convergence preserves pointwise inequalities, we have $0\leq \eta_\infty\leq 1$ a.e. Also
$$
0\leq \alpha_k \int_\Omega \chi_{E_{\alpha_k}} G(|u_{\alpha_k}|)\,dx \leq  \Lambda_\Omega(\alpha_k,c) \leq C
$$
which implies that
$$
0\leq \int_\Omega \chi_{E_{\alpha_k}} G(|u_{\alpha_k}|)\,dx \leq  \frac{C}{\alpha_k }.
$$
Taking limit as $\alpha_k\to\infty$ gives that
$$
\int_\Omega \eta_\infty G(|u_\infty|)\,dx=0.
$$
On the other hand, by the Proposition \ref{bathtube} there exists $E_\infty\subset \Omega$ with $|E_\infty|=c$ such that
$$
\int_\Omega \chi_{E_{\infty}} G(|u_\infty|)\,dx \leq \int_\Omega \eta_\infty G(|u_\infty|)\,dx
$$
and 
$$
\int_\Omega \chi_{E_\infty} G(|u_\infty|)\,dx=0 \implies \chi_{E_\infty} u_\infty=0 \text{ a.e. in } \Omega.
$$
Since $\{\Lambda_\Omega(\alpha,c)\}_{\alpha}$ is bounded uniformly and increasing in $\alpha$, there exists its limit
$$
\lim_{k\to\infty} \Lambda_\Omega(\alpha_k,c):=\Lambda_\Omega(\infty,c)
$$
and
\begin{align*}
\Lambda_\Omega(\infty,c) &= \lim_{k\to\infty} \int_\Omega G(|\nabla u_{\alpha_k}|) + \alpha_k \int_\Omega \chi_{E_{\alpha_k}} G(|u_{\alpha_k}|)\,dx\\
&\geq
\liminf_{k\to\infty} \int_\Omega G(|\nabla u_{\alpha_k}|)\,dx\\
&\geq \int_\Omega G(|\nabla u_\infty|)\,dx\\
&\geq\lambda_\Omega(\infty,E_\infty)\geq\Lambda_\Omega(\infty,c)\\
\end{align*}
where we have used the lower semicontinuity of the modular with respect to the weak convergence and the fact that $u_\infty \chi_{E_\infty}=0$ with $E_\infty\subset \Omega$ and $|E_\infty|=c$.

Therefore, 
\begin{align*}
\Lambda_\Omega(\infty,c) 
&=  \int_\Omega G(|\nabla u_\infty|) \,dx.
\end{align*}
This concludes the proof in the Dirichlet case. The Neumann case is analogous just taking the test functions in $W^{1,G}(\Omega)$ instead of $W^{1,G}_0(\Omega)$. For the Steklov problem the proof is similar by using that    Proposition \ref{compact} additionally gives that $u_k\to u$ strongly in $L^G(\partial \Omega)$ due to the compact embedding, which allows to conclude that $\mathcal{K}(u)=1$.
\end{proof}	

%
%

\section*{Acknowledgments}
AS and AS are members of CONICET, Argentina. BS is a fellow of CONICET. AS and BS are partially supported by ANPCyT PICT 2017-0704 and 2019-03837 and by UNSL under grants PROIPRO 03-2420 and PROICO 03-0720. AS is partially supported by ANPCyT PICT 2017-1119.

\end{document}